\documentclass[12pt]{article}
\usepackage[a4paper, total={7in, 9.5in}]{geometry}
\usepackage{amsmath,amsfonts,amssymb,amsthm}
\usepackage{graphicx}
\usepackage{color}
\usepackage{hyperref}
\usepackage{ragged2e}

\newtheorem{algo}{Algorithm}[section]

\newtheorem{definition}{Definition}[section]
\newtheorem{thrm}{Theorem}[section]
\newtheorem{exmp}{Example}[section]
\newtheorem{note}{Note}
\newtheorem{remark}{Remark}
\usepackage[title]{appendix}
\title{Accuracy Preserving ENO and WENO Schemes using Novel Smoothness Measurement\footnote{Authors acknowledge Council of Scientific and Industrial Research(CSIR), New Delhi, Govt. of India for providing fellowship to Mr. Biswarup Biswas (File no. 09/1045(0009)2K17) and Science and Engineering RB India for its financial support towards procuring computing facility (File No.EMR/2016/000394 ) }}

\author{Biswarup Biswas\footnote{Email: biswarupbiswas.b@res.srmuniv.ac.in} and Ritesh Kumar Dubey\footnote{corresponding author,\; Email: riteshkumar.d@res.srmuniv.ac.in} \\ SRM Research Institute \& Department of Mathematics\\
SRM Institute of Science and Technology}
\date{}
\begin{document}
   \maketitle
\begin{abstract}
	A novel procedure is given for choosing  smoothest stencil to construct less oscillatory ENO schemes. The procedure is further used to define smoothness parameter in the non-linear weights of new WENO schemes. The main significant features of these new ENO and WENO schemes is that they are less oscillatory and achieve their relevant order of accuracy in the presence of critical points in the exact solution. It is shown theoretically as well as computationally in $L^1$ and $L^\infty$ norm. Moreover, computational results are given to show less oscillatory behavior of the new WENO scheme compared to WENO5-JS and WENO5-Z schemes. 
\end{abstract}

\section{Introduction}
Modern arbitrary high order schemes for solving hyperbolic conservation laws use high order reconstruction procedure relying on its non-oscillatory property. In general, such non oscillatory property is achieved by using the idea of adaptive reconstruction. Among the higher order non oscillatory schemes, class of (Essentially Non-oscillatory) ENO and (Weighted Essentially Non-oscillatory) WENO schemes are highly utilized and cited \cite{Shu1988,Jiang1996,Shu1997}. For more on ENO and WENO schemes we refer interested reader to \cite{Shu1997,Shu2009}. Apart from ENO and WENO schemes, a large class of schemes for conservation laws based on a maximum principles also have been popular such as  \cite{Harten1984,ZHANG2010,RBG}. Interested can look into \cite{leveque2002finite,Thomas1999,Whitham1999,Toro2009}.

ENO/WENO schemes for solving conservation laws basically use a ENO/WENO reconstruction in order to reconstruct the conservative variables in the schemes \cite{Shu1997}. The ENO schemes were first proposed in \cite{Harten1987UniformlyI} and their efficient implementation is done in \cite{Shu1988}. ENO schemes are further extended for Hamilton-Jacobi equation in \cite{Osher1991}. ENO and WENO schemes are also used for solving convection dominated problems in \cite{Shu2009}.  Recently, sign stability property of ENO and third order WENO reconstructions is utilized to construct entropy stable schemes called TECNO/mTECNO schemes in \cite{FMT-TECNO,Biswas2017}. However, these schemes are prone to small spurious oscillations in the absence of sufficient diffusion. Very Recently, in \cite{DUBEY2018}, suitable diffsuion matrices are constructed for entropy stable TVD schemes and applied with TECNO schemes to further suppress the induced spurious oscillations. ENO and WENO schemes lack mathematical proofs on stability in general though it is supported through a wide range of excellent numerical evidences. It is well known that classical ENO reconstruction is computationally very costly for the reason of continuous comparison of the divided differences \cite{Shu1997}. Also a $k^{th}$ order reconstruction requires all divided differences up to $k^{th}$ order. The complications increase with the accuracy order of the ENO reconstruction. Additionally, the stencil choosing process in ENO reconstruction can experience an error in choosing the smoothest stencil. This is due to the lead of rounding up error when the divided differences involved in the comparison are equal in value. 

These above mentioned shortcomings of ENO reconstruction motivate us to search for a new algorithm. It is intended here to develop a new strategy of ENO stencil choosing process without compromising the non-oscillatory property of the scheme. We also extend the idea to construct a WENO reconstruction. The main objective in this article is to modify the ENO/WENO procedure in order to get a better reconstruction and then apply it to construct accuracy preserving ENO/WENO schemes for solving conservation laws. In this work, the focus is on finite volume formulation, however, one can proceed with the finite difference formulation too.

\par For the sake of completeness and ease of the presentation, a brief overview of ENO interpolation is given in Section \ref{ClassicalENO} which is also essential to propose the new algorithm. Next, in Section \ref{new_interpolation} we have introduced the novel stencil choosing algorithm for ENO interpolation and ENO reconstruction. New WENO reconstruction using length based non-linear weights and accuracy is given in Section \ref{WEno_re}. Numerical results for benchmark 1D and 2D test problems are given in section \ref{num_results} followed by Appendix \ref{A1}. Numerical results clearly show less oscillatory behavior with improved formal accuracy of constructed ENO and WENO schemes.    
   
\section{ENO Interpolation}\label{ClassicalENO}
Let us consider a piecewise continuous function $v(x)$ in $\Omega \subset \mathbb{R}$. The domain $\Omega$  is partitioned with the grids $\{x_i\}$, $i\in\mathbb{Z}$ and the point values are given by $v_{i}=v(x_{i})$. The $k$-th order ENO interpolation procedure in an arbitrary interval $I_{i}:=[x_{i-\frac{1}{2}}, x_{i+\frac{1}{2}}]$, where $x_{i+\frac{1}{2}}=\frac{x_{i}+x_{i+1}}{2}$ consists of two steps. In the first step (also known as ENO stencil choosing step) stencil with $k$ consecutive points $S_{pref}=\{x_{i-r},\,\dots,\,x_{i},\,\dots,\, x_{i-r+k-1}\}$ is chosen where $r\in\{0,1,\dots, \left(k-1\right)\}$. In the next step, the unique $\left(k-1\right)$-th degree polynomial passing through $S_{pref}$ is interpolated. The classical ENO stencil choosing process use up to $k$-th order divided differences $\{v[x_{i},x_{i+1},\, \dots,\, x_{i+r}]\}_i$ where the first order divided difference is $v[x_{i}]=v_i$. The ENO stencil choosing procedure is given by following (see also \cite{Shu1997,Fjordholm2013}).
\begin{algo}\label{algo:old}
	Consider the point values $v_{i-k+1},\,\dots,\,v_{i},\,\dots,\, v_{i+k-1}$.
	\begin{itemize}
		\item[(i)] Set $r=0$.
		\item[(ii)] Repeat the following for $j=1,2,\,\dots,\,k-1.$\\
		$\begin{cases}
		\text{If } |v[x_{i-r-1},\, \dots,\, x_{i-r+j-1}]|<|v[x_{i-r},\, \dots,\, x_{i-r+j}]|,\\
		 \text{Set}\,\, r=r+1.
		\end{cases}$
		\item[(iii)] The preferred stencil is $S_{pref}=\{x_{i-r},\,\dots,\,x_{i},\,\dots,\, x_{i-r+k-1}\}$
	\end{itemize}
\end{algo}
%\begin{algo}
%		Consider the point values $v_{i-k+1},\,\dots,\,v_{i},\,\dots,\, v_{i+k-1}$.
%\begin{algorithmic}[1]
%	\STATE Set $r=0$.
%	\FOR {j=1,2,...,k-1} 
%	\IF {$|v[x_{i-r-1},\, \dots,\, x_{i-r+j-1}]|<|v[x_{i-r},\, \dots,\, x_{i-r+j}]|$} 
%	\STATE {$r$ $\gets$ $r+1$}
%	\ENDIF
%	\ENDFOR
%	\STATE The preferred stencils are $S_{pref}=\{x_{i-r},\,\dots,\,x_{i},\,\dots,\, x_{i-r+k-1}\}$
%\end{algorithmic}
%\end{algo}
Basically, Algorithm \ref{algo:old} gives us $k$ smoothest consecutive points containing $x_i$ form  the set of $(2k-1)$ points \{$x_{i-k+1},\,\dots,\,x_{i},\,\dots,\, x_{i+k-1}$\}\cite{Shu1988}. The ENO interpolation procedure is finished by performing the interpolation on $S_{pref}$, which is given by \begin{equation}
P_{i}(x)=v[x_{i-r}]+\sum_{j=1}^{k-1}v[x_{i-r},\, \dots,\, x_{i-r+j}]\prod_{l=0}^{j-1}(x-x_{i-r+l}).
\end{equation}
Finally, the polynomial $P_{i}(x)$ is used to reconstruct the following cell interface  values
\begin{eqnarray*}
	v_{i+\frac{1}{2}}^{-}:=P_{i}(x_{i+\frac{1}{2}}),
	\\v_{i-\frac{1}{2}}^{+}:=P_{i}(x_{i-\frac{1}{2}}).
\end{eqnarray*}
As mentioned that there are some shortcomings in ENO stencil choosing algorithm  \ref{algo:old}. Following examples demonstrate that at times it may not end up choosing the smoothest stencil.
\begin{exmp}\label{exmp1}
	Let us consider the data set from the sin function $$\displaystyle \left \{(0,0),\left(\frac{\pi}{4},\frac{1}{\sqrt{2}}\right),\left(\frac{\pi}{2},1\right),\left(\frac{3\pi}{4},\frac{1}{\sqrt{2}}\right),(\pi,0)\right\}.$$
\end{exmp}
We start ENO procedure in algorithm  \ref{algo:old}, from the point $\frac{\pi}{2}$. Observe that second divided differences $v[\frac{\pi}{4},\frac{\pi}{2}]$ and $v[\frac{\pi}{2},\frac{3\pi}{4}]$ are equal in absolute value. Therefore, Algorithm  \ref{algo:old} can choose either of the stencil. It can be seen that in both the cases end up with the final preferred smoothest stencil $S_{pref}=\{\frac{\pi}{4},\frac{\pi}{2},\frac{3\pi}{4}\}$ because  \begin{equation*}
\left| v\left[0,\frac{\pi}{4},\frac{\pi}{2}\right]\right| > \left|v\left[\frac{\pi}{4},\frac{\pi}{2},\frac{3 \pi}{4}\right]\right| <\left|v\left[\frac{\pi}{2},\frac{3 \pi}{4},\pi\right]\right|.
\end{equation*}
However, it is not the always be the case. Consider a slight change in data set of the Example \ref{exmp1}.
\begin{exmp} \label{exmp2}
	Consider the point values $$\displaystyle \left \{(0,0),\left(\frac{\pi}{4},\frac{1}{\sqrt{2}}\right),\left(\frac{\pi}{2},1\right),\left(\frac{3\pi}{4},\frac{1}{\sqrt{2}}\right),(\pi,\frac{1}{2})\right \}.$$
\end{exmp} 
Here also, in the first level comparison, the divided differences  $v[\frac{\pi}{4},\frac{\pi}{2}]$ and $v[\frac{\pi}{2},\frac{3\pi}{4}]$ are equal in value and Algorithm  \ref{algo:old} can pick any of the stencil between  $\{\frac{\pi}{4},\frac{\pi}{2}\}$ and $\{\frac{\pi}{2},\frac{3\pi}{4}\}$. Note that in the comparison among third divided differences, following relation holds \begin{equation}\label{eg2p}
\left| v\left[0,\frac{\pi}{4},\frac{\pi}{2}\right]\right|>\left|v\left[\frac{\pi}{4},\frac{\pi}{2},\frac{3 \pi}{4}\right]\right|>\left|v\left[\frac{\pi}{2},\frac{3 \pi}{4},\pi\right]\right|.
\end{equation}
It can be concluded that, the choice $\{\frac{\pi}{2},\frac{3\pi}{4}\}$  as smoothest stencil in the first level comparison will end up with picking final stencil as $S_{pref}=\{\frac{\pi}{2},\frac{3\pi}{4},\pi\}$ which is indeed smoothest with smallest third divided difference. On the contrary, choice of stencil $\{\frac{\pi}{4},\frac{\pi}{2}\}$ in first level leads to get $S_{pref}=\{\frac{\pi}{4},\frac{\pi}{2},\frac{3\pi}{4}\}$  as final smoothest stencil though it has higher value for third divided difference in \eqref{eg2p}. 
\par An another related problem of choosing lesser smooth stencil can cause during computation when the truncation error lead in the comparison if divided differences are equal in absolute value.. In the next section an novel adaptive stencil choosing process is proposed to sort out some of these problems.
 
\section{Novel ENO interpolation and ENO reconstruction}\label{new_interpolation}
It can be observed that the classical ENO procedure utilizes the behavior of data to choose the smoothest stencil. In this work, the novel $k^{th}$ order ENO interpolation procedure chooses the smoothest polynomial instead smoothest stencil by analyzing the smoothness of all $k$ consecutive interpolating polynomials using $2k-1$ data set. The notion of smoothness of polynomial is inspired by the work of Rivlin in \cite{Rivlin1959} where the length of a curve is used in order to define a smoother polynomial. More precisely, given a family of polynomials the smoothest polynomial is the one which has smallest curve length.  
We define the following.
\begin{definition} \label{defn1}
	Given two polynomials $P^1(x)$ and $P^2(x)$ in $[a,\,b]\in \mathbb{R}$, we say $P^1$ is \textbf{smoother} than $P^2$ if \begin{equation}
	\mathcal{L}_{[a,\,b]}(P^1)\leq\mathcal{L}_{[a,\,b]}(P^2),	
	\end{equation}
where $\mathcal{L}_{[a,b]}$ is the length of a given polynomial in $[a,b]$ defined by\begin{equation}\label{length}
\mathcal{L}_{[a,\,b]}(P)=\int_{a}^{b}\sqrt{1+{P'(x)}^2} dx.
\end{equation}
\end{definition}
Then, using the smoothness of polynomial, we propose the following stencil choosing algorithm.
\begin{algo}[New stencil choosing process]\label{algo:new} Consider the point values $v_{i-k+1}$, $\dots$, $v_{i}$, $\dots$, $v_{i+k-1}$. Let $P^j$ $(j=0,1,2,...,k-1)$ is the polynomial which is interpolating $v_{i-k+1+j},\,\dots,\,v_{i+j}$.
	\begin{itemize}
		\item[(i)] Calculate the lengths $\mathcal{L}_{\left[x_{i-1/2},\,x_{i+1/2}\right]}(P^j)$ for $j=0,1,\,\dots,\, k-1$ 
		\item[(ii)] Find $r$ such that $\mathcal{L}_{\left[x_{i-1/2},\,x_{i+1/2}\right]}(P^r)=\displaystyle\min_{j\in \{{0,\,\dots,\,k-1}\}}\mathcal{L}_{\left[x_{i-1/2},\,x_{i+1/2}\right]}(P^j)$
		\item[(iii)] The preferred stencils are $S_{pref}=\{x_{i-r},\,\dots,\,x_{i},\,\dots,\, x_{i-r+k-1}\}$
	\end{itemize}
\end{algo}
Note that, the `smoothest polynomials' which in general differ from the polynomial on `smoothest stencil' is chosen in step (ii), therefore Algorithm \ref{algo:new} may leads to picking stencil different from the one picked by Algorithm \ref{algo:old}. For example, again consider the data of Example \ref{exmp2} with the Algorithm \ref{algo:new}. The lengths of second degree polynomials are
\begin{equation*}
{\mathcal{L}_{[0,\, \frac{\pi}{2}]}(P^0(x))=1.90627, \mathcal{L}_{[\frac{\pi}{4}, \,\frac{3\pi}{4}]}(P^1(x))=1.7062,	\mathcal{L}_{[\frac{\pi}{2}, \,\pi]}(P^2(x))=1.65115.}
\end{equation*}
Observed that the new Algorithm \ref{algo:new} picks the polynomial $P^{2}(x)$ as smoothest which is natural as the corresponding third divided difference is smallest in \eqref{eg2p}.\\
 The above stencil choosing process is quite simple and in case of non uniqueness of  minimal length polynomial, index $r$ in $(ii)$ step of Algorithm \ref{algo:new} can be chosen corresponding to any of the same length polynomials. The total number of comparison in (ii) of Algorithm \ref{algo:old} and in (ii) of Algorithm \ref{algo:new} are equal. However, for a $k$-th order reconstruction in Algorithm \ref{algo:old} we require all divided differences up to $k-1$ degree whereas in Algorithm \ref{algo:new} the only quantity to calculate is the lengths of polynomials.
%\section{Novel ENO reconstruction}\label{Eno_re}
\par ENO reconstruction is closely related to ENO interpolation. Let us recall the ENO reconstruction procedure from \cite{Shu1997}. Unlike point values $v_{i}$ in ENO interpolation, here following cell averages are considered. \begin{equation} 
\bar{v_{i}}=\frac{1}{\Delta x_{i}}\int_{x_{i-\frac{1}{2}}}^{x_{i+\frac{1}{2}}} v(z) dz.
\end{equation}
$k^{th}$ order ENO reconstruction finds a polynomial $p_{i}(x)\in I_{i},\,(i\in \mathbb{Z})$ such that for $x\in I_{i}$\begin{equation}\label{avraccr}
p_{i}(x)=v(x)+O({\Delta x}^k).
\end{equation}
Let $P^r(x), \, (r=0,1,...,k-1)$ interpolates the primitive function $V(x)$ defined by \begin{equation}
V(x)=\int_{\infty}^{x}v(z) dz,
\end{equation} 
on $k+1$ preferred  stencils \begin{equation}
S_{pref}=\{x_{i-r-\frac{1}{2}},\,\dots,\,x_{i+\frac{1}{2}},\,\dots,\, x_{i-r+k-\frac{1}{2}}\} .
\end{equation}
Then $p^r_{i}(x)$ is defined as \begin{equation}\label{ENO_pol}
p^r_{i}(x)=\frac{d}{dx}P^r(x).
\end{equation}
It is shown in \cite{Shu1997} that such $p^r_{i}(x)$ satisfies \eqref{avraccr}. Here, an idea similar to Algorithm \eqref{algo:old} is used to find the preferred stencil $S_{pref}$ (for details see \cite{Fjordholm2013}). We change the stencil choosing process similar to the interpolation procedure. We choose $r$ such that \begin{equation}
\mathcal{L}_{\left[x_{i-1/2},\,x_{i+1/2}\right]}(p^r_i)=\displaystyle\min_{j\in \{{0,\,\dots,\,k-1}\}}\mathcal{L}_{\left[x_{i-1/2},\,x_{i+1/2}\right]}(p^j_i).
\end{equation}
Finally define, \begin{eqnarray*}
	v_{i+\frac{1}{2}}^{-}:=p^r_{i}(x_{i+\frac{1}{2}}),
	\\v_{i-\frac{1}{2}}^{+}:=p^r_{i}(x_{i-\frac{1}{2}}).
\end{eqnarray*}
\section{Novel WENO Reconstruction}\label{WEno_re}
\justify
As WENO reconstruction is more relevant than WENO interpolation in solving conservation laws \cite{Shu1988,Shu2009}, therefore in this section modification is given only in the WENO reconstruction procedure. However, WENO interpolation procedure can be modified similarly. A WENO reconstruction is done by taking a combination of all $k^{th}$ order ENO polynomials $p^r$ in section   \ref{new_interpolation} with suitable non-linear weights in order to get $2k-1$ order approximation. More precisely, on $2k-1$ point stencil $\{x_{i-k+\frac{1}{2}},\,\dots,\,x_{i+\frac{1}{2}},\,\dots,\, x_{i+k-\frac{1}{2}}\}$, the classical WENO reconstruction is defined as \begin{subequations}\label{WENO_Re}
	\begin{equation}
	v_{i+\frac{1}{2}}^{-}=\sum_{j=0}^{k-1} \omega_j p^{j}_i(x_{i+\frac{1}{2}}),
	\end{equation}
	\begin{equation}
	v_{i-\frac{1}{2}}^{+}=\sum_{j=0}^{k-1} \tilde{\omega}_j p^{j}_i(x_{i-\frac{1}{2}}).
	\end{equation}
\end{subequations}
where non-linear weights $w_j, \tilde{w}_j$ are given by
\begin{subequations}
	\begin{equation}
	\omega_j=\frac{\alpha_j}{\sum_{p=0}^k\alpha_p},\, \, \tilde{\omega}_j=\frac{\tilde{\alpha}_j}{\sum_{p=0}^k\tilde{\alpha}_p},
	\end{equation}
	with 
	\begin{equation}\label{beta_j}
	\alpha_j=\frac{\gamma_{j}}{(\epsilon+\beta_j)^2},\,\,\tilde{\alpha}_j=\frac{\tilde{\gamma}_{j}}{(\epsilon+\beta_j)^2}.
	\end{equation}
\end{subequations}
The constants $\gamma_{j}$ and $\tilde{\gamma}_j$ are chosen such that 
\begin{equation*}
\sum_{j=0}^{k-1} \gamma_j p^{j}_i(x_{i+\frac{1}{2}})-v(x_{i+\frac{1}{2}})=O(h^{2k-1}),
\end{equation*}
and
\begin{equation*}
\sum_{j=0}^{k-1} \tilde{\gamma}_j p^{j}_i(x_{i-\frac{1}{2}})-v(x_{i-\frac{1}{2}})=O(h^{2k-1}).
\end{equation*}
The smoothness parameters $\beta_j$'s in \eqref{beta_j} are given by\begin{equation}\label{beta}
\beta_j=\sum_{l=1}^{k}\int_{x_{i-\frac{1}{2}}}^{x_{i+\frac{1}{2}}}{\Delta x}^{2l-1}\left(\frac{d^l}{dx^l}p^{j}_i(x)\right)^2 dx, \,(j=0,1,...,k-1)
\end{equation} 
It is needed to be emphasized that the original WENO schemes suffers from the accuracy drops at critical points and extensive amount of work is done to redefine the smoothness paremeter $\beta_j$'s given in \eqref{beta} to achieve improved accuracy at critical points \cite{borges2008improved,henrick2005mapped}.The improvement using transformation in \cite{borges2008improved} is interesting as resulting fifth order WENO scheme (known as WENO-Z scheme) retained its formal accuracy at critical points. The transformed smoothness indicators in \cite{borges2008improved} denoted as $\beta^Z_{j}$ and defined by 
\begin{equation}\label{betaZ}
\frac{1}{\beta^Z_{j}}=\left(1+\frac{\tau_5}{\beta_j+\epsilon}\right), \tau_5:=|\beta_0-\beta_2|; j=0,1,2
\end{equation}
 There are other choices of smoothness parameter $\beta_j$'s in literature, some of them can be found in \cite{serna2004power,fan2014new,ha2013improved,kim2016modified,Rathan2017modified}. However, most of them are developed in order to get better accuracy and are mostly a  translation of \eqref{beta}. 

Here in order to design WENO reconstruction we modify the smoothness parameter $\beta_j$'s by utilizing the idea of polynomial length as in new ENO reconstruction. In particular, it is done by defining $\beta_j$'s based on the length of the ENO polynomials. For the fifth order $(i.e., k=2)$ WENO reconstruction new smoothness parameters defined as
\begin{equation}\label{newbeta}
\beta_j=\left(\mathcal{L}_{[x_{i-\frac{1}{2}},\,x_{i+\frac{1}{2}}]}(p^j_i)\right)^2, \,(j=0,1,2).
\end{equation} 
An another variant can be obtained by transforming the smoothness indicator in \eqref{newbeta} by using \eqref{betaZ}. Thanks to the length based $\beta_j$'s, the resulting WENO scheme retains the higher order accuracy in the smooth region of data including critical points as shown below.
\subsubsection*{Accuracy of WENO Reconstruction}
Observed that the classical fifth order WENO reconstruction satisfy the following accuracy relation\begin{equation*}
v^{\pm}_{i+\frac{1}{2}}=v(x_{i+\frac{1}{2}})+O({\Delta x}^{5}).
\end{equation*}
In order to attain the same accuracy, smoothness parameter in \eqref{newbeta} must satisfy the sufficient condition \cite{Jiang1996}
\begin{equation}\label{suff_beta}
\beta_j=D(1+O({\Delta x}^2)), \,(j=0,1,2),
\end{equation}
where $D$ may be depend on $\Delta x$.

\begin{thrm}
	Smoothness parameters $\beta_j$(j=0,1,2) defined in \eqref{newbeta} satisfy  \eqref{suff_beta}.
\end{thrm}

 \begin{proof}
 	Let $c_i=\bar{v}_{i-1}-2\bar{v}_i+\bar{v}_{i+1}$ for all $i\in \mathbb{Z}$. Then,
Taylor series expansion of the lengths (see Appendix \ref{A1}) result the following
\begin{subequations}\label{taylor}
	{\small\begin{equation}
		\mathcal{L}_{[x_{i-\frac{1}{2}},\,x_{i+\frac{1}{2}}]}(p_i^0)
		=\begin{cases}
		\sqrt{1+{v'}^2}\Delta x +\displaystyle \frac{{v''}^2-8{v'}{v'''}-8{v'}^3{v'''}}{24 (1+{v'}^2)^{3/2}}{\Delta x}^3+O({\Delta x}^4),& \hspace{-5pt}\text{$c_{i-1}\neq0$} \\
		\sqrt{1+{v'}^2}\Delta x +\displaystyle\frac{{v'}{v'''}}{3 \sqrt{1+{v'}^2}}{\Delta x}^3+O({\Delta x}^4),& \hspace{-10pt}\text{otherwise}
		\end{cases}
		\end{equation}}
	{\small\begin{equation}
		\mathcal{L}_{[x_{i-\frac{1}{2}},\,x_{i+\frac{1}{2}}]}(p_i^1)=\begin{cases}
		\sqrt{1+{v'}^2}\Delta x +\displaystyle\frac{{v''}^2+4{v'}{v'''}+4{v'}^3{v'''}}{24 (1+{v'}^2)^{3/2}}{\Delta x}^3+O({\Delta x}^5),& \hspace{-5pt}\text{$c_{i}\neq0$}\\
		\sqrt{1+{v'}^2}\Delta x +\displaystyle\frac{{v'}{v'''}}{6 \sqrt{1+{v'}^2}}{\Delta x}^3+O({\Delta x}^5),& \hspace{-10pt}\text{otherwise}
		\end{cases}
		\end{equation}}
	{\small\begin{equation}
		\mathcal{L}_{[x_{i-\frac{1}{2}},\,x_{i+\frac{1}{2}}]}(p_i^2)=\begin{cases}
		\sqrt{1+{v'}^2}\Delta x +\displaystyle\frac{{v''}^2-8{v'}{v'''}-8{v'}^3{v'''}}{24 (1+{v'}^2)^{3/2}}{\Delta x}^3+O({\Delta x}^4),&\hspace{-5pt} \text{$c_{i+1}\neq0$}\\
		\sqrt{1+{v'}^2}\Delta x +\displaystyle\frac{{v'}{v'''}}{3 \sqrt{1+{v'}^2}}{\Delta x}^3+O({\Delta x}^4),& \hspace{-10pt}\text{otherwise}
		\end{cases}
		\end{equation}}
\end{subequations}
Clearly, from \eqref{taylor} and \eqref{newbeta} we can conclude\eqref{suff_beta}. 
 \end{proof}
\begin{note}
If $v'=0$, the leading terms in \eqref{taylor} still remain non-zero which prevent the accuracy drop on critical points.
\end{note}

\section{Numerical results}\label{num_results}
Semi-discrete scheme to solve the conservation laws \begin{equation}\label{conservationlaws}
\mathbf{u}_t+\mathbf{f}(\mathbf{u})_x=0
\end{equation}
is given by
\begin{subequations}\label{semidiscrete}
	\begin{equation}
	\frac{d\mathbf{u}_i(t)}{dt}= L(\mathbf{u}_i),
	\end{equation}
	\begin{equation}
	L(\mathbf{u}_i)=-\frac{1}{\Delta x}\left[{\hat{\mathbf{f}}}_{i+\frac{1}{2}} - {\hat{\mathbf{f}}}_{i-\frac{1}{2}}\right],
	\end{equation}
\end{subequations}
where $\hat{\mathbf{f}}_{i+1/2}$ is numerical flux reconstructed from point values $\{\mathbf{f}(\mathbf{u}_i)\}$. In system case, component wise reconstruction defined in \cite{Shu1997} is used which is method is simple and computationally much more efficient compared to characteristic wise reconstruction method as it is free of calculating several Jacobian matrices. On applying the flux splitting 
\begin{equation}
\mathbf{f}=\mathbf{f}^++\mathbf{f}^-,
\end{equation}
such that $\dfrac{d\mathbf{f}}{d\mathbf{u}}^+\geq 0$ and $\dfrac{d\mathbf{f}}{d\mathbf{u}}^-\leq 0$.
Next for $m^{th}$ component $f^{(m)\pm}$ of $\mathbf{f}^{\pm}$, take $\bar{v}_{i}=f^{(m)\pm}_i$ and then apply ENO/WENO reconstruction procedure to get the $m^{th}$ component of ${\hat{\mathbf{f}}}_{i+\frac{1}{2}}^{\pm}$ defined as\begin{equation}
{\hat{f}^{(m)\pm}}_{i+\frac{1}{2}}=v_{i+\frac{1}{2}}^{\mp}.
\end{equation} The numerical flux is obtained by \begin{equation}
{\hat{\mathbf{f}}}_{i+\frac{1}{2}}={\hat{\mathbf{f}}}_{i+\frac{1}{2}}^{+}+{\hat{\mathbf{f}}}_{i+\frac{1}{2}}^{-}.
\end{equation}
\par The following third order SSP Runge Kutta time discretization used to fully discretize the semi-discrete scheme \eqref{semidiscrete}
\begin{equation}
\begin{cases}\label{3rdSSP}
\mathbf{u}^{(1)}=  \mathbf{u}^n+\Delta t L(\mathbf{u}^n),\\
\mathbf{u}^{(2)}= \frac{3}{4}\mathbf{u}^n+\frac{1}{4} \mathbf{u}^{(1)}+\frac{1}{4}\Delta t L(\mathbf{u}^{(1)}),\\
\mathbf{u}^{(n+1)} = \frac{1}{3}\mathbf{u}^n+\frac{2}{3} \mathbf{u}^{(2)}+\frac{2}{3}\Delta t L(\mathbf{u}^{(2)}).
\end{cases}
\end{equation}

\par For purity of this section, the numerical results are presented where one can calculate the lengths exactly, that is, using third order ENO reconstruction and fifth order WENO reconstruction. The new third order ENO scheme using the reconstruction defined in section \ref{new_interpolation} and the fifth order WENO scheme using the reconstruction in section \ref{WEno_re} are termed as  ENO3-$\mathcal{L}$ and WENO5-$\mathcal{L}$ respectively. We call the new WENO scheme  WENO5-Z$\mathcal{L}$ obtained by using the translation of new smoothness indicator through \eqref{betaZ}.
\begin{remark} The length of a polynomial curve can be calculated explicitly if the order of the polynomial is less than or equal to three. However, for higher order polynomial one can approximate the length in very efficient way by calculating the integration involved in \eqref{length} numerically  using a Gauss quadrature method. 
\end{remark}
\subsection{Scalar Conservation Laws}
\subsubsection*{Linear advection equation (Accuracy test)}
The accuracy of ENO3-$\mathcal{L}$ and WENO5-$\mathcal{L}$ are tested for linear advection equation,
\begin{equation}\label{transport}
u_t+u_x=0,
\end{equation}
 using the initial data
  \begin{equation}\label{IC2} u(x,0)=sin^4(\pi x)\; \text{in periodic domain}\; [0,1]
 \end{equation} at a final time $t=0.55$. The $L^\infty$ and $L^1$ errors are presented in Table \ref{tab:ENO3} and Table \ref{tab:ENOL} for schemes ENO3 (third order ENO) and ENO3-$\mathcal{L}$ respectively which shows the better accuracy by ENO3-$\mathcal{L}$. ENO3-$\mathcal{L}$ achieves its third order of accuracy.  It is well known that the classical fifth order WENO scheme fails to achieve its desire accuracy which can also be seen in Table \ref{tab:WENOJS}. Table \ref{tab:WENOL} shows that the WENO5-$\mathcal{L}$ scheme achieve its fifth order accuracy.
 \begin{table}[htb!]
 	\centering
 	\begin{tabular}{|c|c|c|c|c|}
 		\hline N  & $L^\infty$ error  &  Rate &    $L^1$ error    & Rate \\ 
 		\hline 20 & 0.137523991504696 &  ...  & 0.124370622299486 & ... \\ 
 		\hline 40 & 0.022484345494414 &  2.61 & 0.019284894776954 & 2.69 \\ 
 		\hline 80 & 0.004684412057962 &  2.26 & 0.003665119042135 & 2.40 \\ 
 		\hline 160 & 0.000806458682295 &  2.54 & 0.000525027379226 & 2.80 \\ 
 		\hline 320 & 0.000216993609355 &  1.89 & 0.000093980577132 & 2.48 \\ 
 		\hline 
 	\end{tabular} 
 \caption{Convergence rate of ENO3.}
\label{tab:ENO3}
 	% The solution of test 9 at T=0.55 for CFL=0.10. 
 	% It was exucuted through C:\Users\Admin\Desktop\ENO_WENO-L_Code\LinearBurger - ENO-L on 04-Jul-2018
 \end{table}
\begin{table}[htb!]
	\centering
	\begin{tabular}{|c|c|c|c|c|}
		\hline N  & $L^\infty$ error  &  Rate &    $L^1$ error    & Rate \\ 
		\hline 20 & 0.138889189045624 &  ...  & 0.123166131013945 & ... \\ 
		\hline 40 & 0.022407420203306 &  2.63 & 0.018438218920095 & 2.74 \\ 
		\hline 80 & 0.003934923788340 &  2.51 & 0.003355520207242 & 2.46 \\ 
		\hline 160 & 0.000537950648222 &  2.87 & 0.000438182124548 & 2.94 \\ 
		\hline 320 & 0.000080526081835 &  2.74 & 0.000051962761009 & 3.08 \\ 
		\hline 
	\end{tabular} 
\caption{Convergence rate of ENO3-$\mathcal{L}$.}
\label{tab:ENOL}
	% The solution of test 9 at T=0.55 for CFL=0.10. 
	% It was exucuted through C:\Users\Admin\Desktop\ENO_WENO-L_Code\LinearBurger - ENO-L on 04-Jul-2018
\end{table}

\begin{table}[htb!]
	\centering
	\begin{tabular}{|c|c|c|c|c|}
		\hline N  & $L^\infty$ error  &  Rate &    $L^1$ error    & Rate \\ 
		\hline 20 & 0.094055534844950 &  ...  & 0.080917224223634 & ... \\ 
		\hline 40 & 0.007890539989593 &  3.58 & 0.006756269739283 & 3.58 \\ 
		\hline 80 & 0.001365228509699 &  2.53 & 0.000671991117245 & 3.33 \\ 
		\hline 160 & 0.000127329078137 &  3.42 & 0.000036747388665 & 4.19 \\ 
		\hline 320 & 0.000012043415363 &  3.40 & 0.000001992307404 & 4.21 \\ 
		\hline 
	\end{tabular} 
	\caption{Convergence rate of WENO5-JS.}
	\label{tab:WENOJS}
	% The solution of test 9 at T=0.55 for CFL=0.25. 
	% It was exucuted through C:\Users\Admin\Desktop\ENO_WENO-L_Code\LinearBurger on 04-Feb-2018
\end{table}
\begin{table}[htb!]
	\centering
	\begin{tabular}{|c|c|c|c|c|}
		\hline N  & $L^\infty$ error  &  Rate &    $L^1$ error    & Rate \\ 
		\hline 20 & 0.095122975489415 &  ...  & 0.090373562413356 & ... \\ 
		\hline 40 & 0.006873815741128 &  3.79 & 0.004734156253646 & 4.25 \\ 
		\hline 80 & 0.000215125442465 &  5.00 & 0.000196079943647 & 4.59 \\ 
		\hline 160 & 0.000007758132987 &  4.79 & 0.000006404741272 & 4.94 \\ 
		\hline 320 & 0.000000251699574 &  4.95 & 0.000000198826305 & 5.01 \\ 
		\hline 
	\end{tabular} 
	\caption{Convergence rate of WENO5-$\mathcal{L}$.}
	\label{tab:WENOL}
	% The solution of test 9 at T=0.55 for CFL=0.25. 
	% It was exucuted through C:\Users\Admin\Desktop\ENO_WENO-L_Code\LinearBurger on 04-Feb-2018
\end{table}

\subsubsection*{Burgers equation}
This non-linear scalar conservation law is considered in order to show the non-oscillatory property of the proposed schemes. The following smooth initial condition is taken in $[-1,1]$, \begin{equation}\label{tp:shu_osher_burger}
u(x,0)=1+\frac{1}{2}\sin (\pi x)
\end{equation}
which produce a shock at time $\displaystyle t=\frac{2}{\pi}$. Solutions are calculated at $t=1$ by using ENO3-$\mathcal{L}$ and WENO5-$\mathcal{L}$, WENO5-Z$\mathcal{L}$ and plotted in Figure \ref{Burgert3} along with the solution by ENO3 and WENO5-JS and WENO5-Z respectively. No big oscillations are observed by the schemes. In Figure \ref{Burgert3} it can be seen that WENO5-$\mathcal{L}$ and  WENO5-Z$\mathcal{L}$ provide slightly sharper resolution to discontinuity than WENO5-JS and WENO5-Z. Convergence rates by the scheme ENO3-$\mathcal{L}$ and WENO5-$\mathcal{L}$ are also calculated at time $T=\frac{2}{\pi}$ and presented in Table \ref{tab:burger_Acc_ENOL} and Table \ref{tab:burger_Acc_WENOL} respectively. 
\begin{figure}
	\centering
	\includegraphics[scale=0.55]{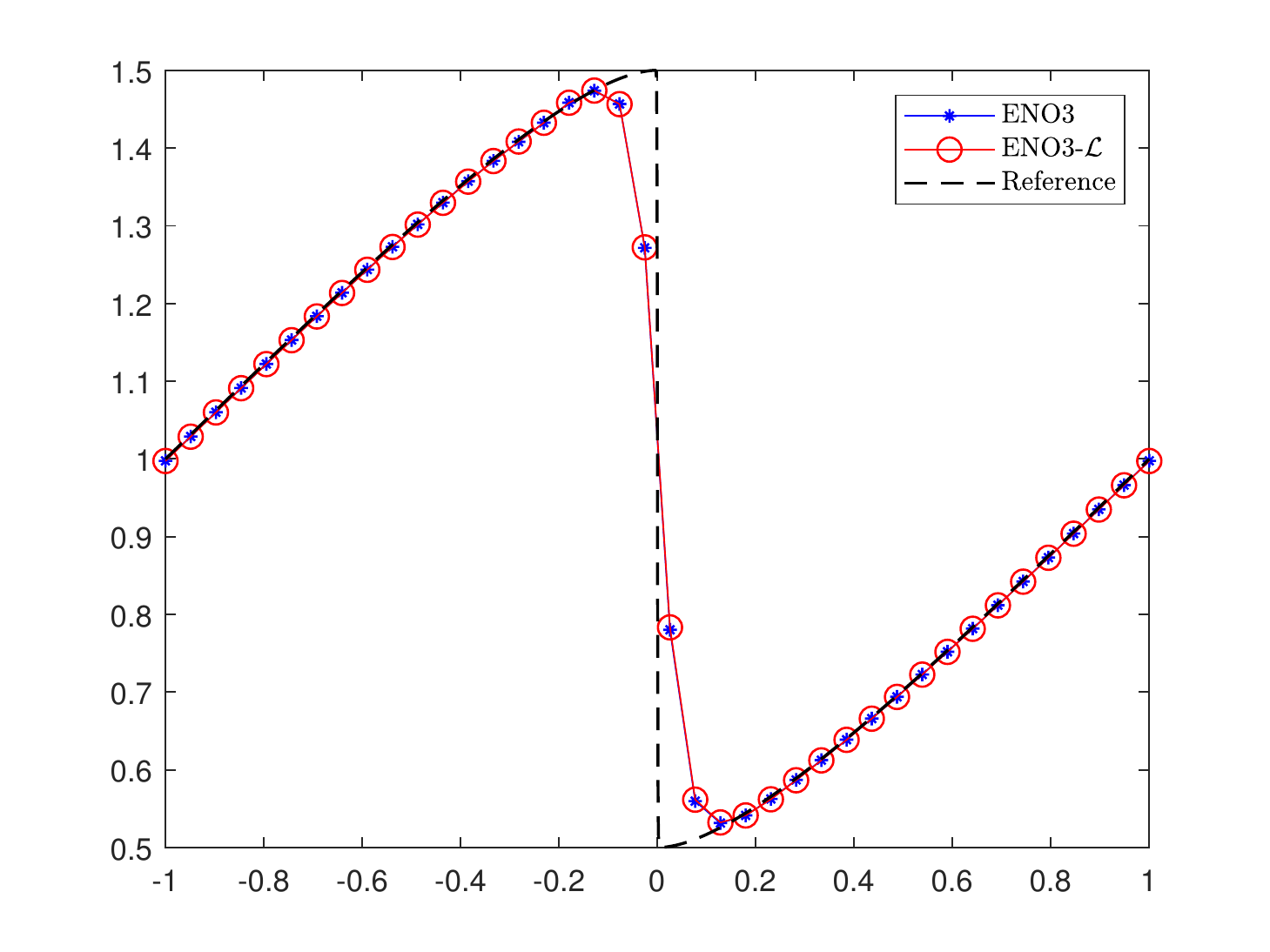}
	\includegraphics[scale=0.55]{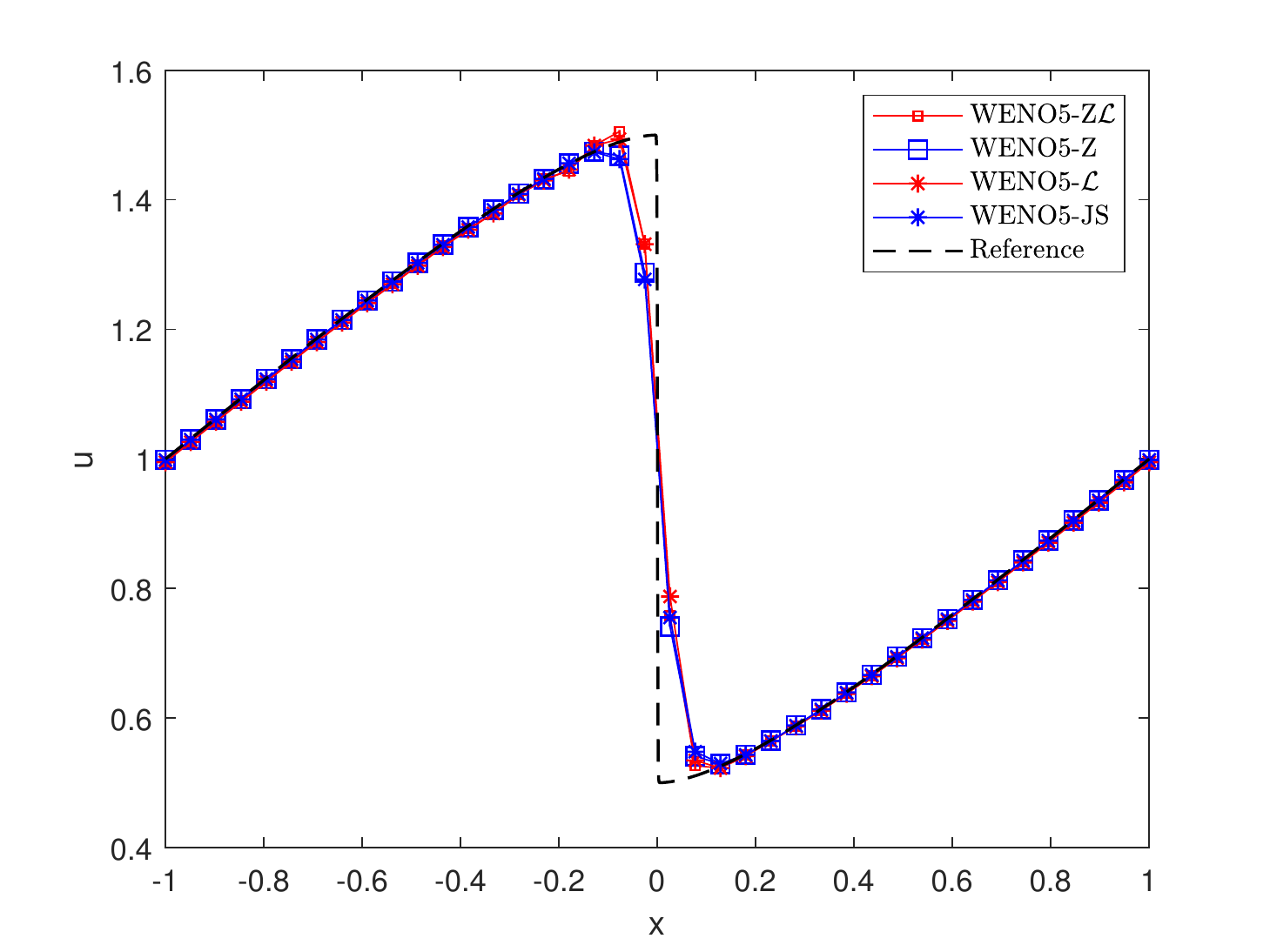}
	\caption[Solution of Burger equation by ENO-$\mathcal{L}$ and WENO-$\mathcal{L}$.]{Solution of Burger equation at $t=1$ with initial ($t=0$) given data \eqref{tp:shu_osher_burger}.}
	\label{Burgert3}
\end{figure}
\begin{table}[htb!]
	\centering
	\begin{tabular}{|c|c|c|c|c|}
		\hline N  & $L^\infty$ error  &  Rate &    $L^1$ error    & Rate \\ 
		\hline 20 & 0.004610240510194 &  ...  & 0.002825622061104 & ... \\ 
		\hline 40 & 0.000742098084899 &  2.64 & 0.000352048256755 & 3.00 \\ 
		\hline 80 & 0.000142409790295 &  2.38 & 0.000042142988693 & 3.06 \\ 
		\hline 160 & 0.000030147895097 &  2.24 & 0.000005499473205 & 2.94 \\ 
		\hline 320 & 0.000006255477703 &  2.27 & 0.000000707089805 & 2.96 \\ 
		\hline 
	\end{tabular} 
	\caption{Convergence rate of ENO3-$\mathcal{L}$.}
	\label{tab:burger_Acc_ENOL}
	% The solution of test 3 at T=0.16 for CFL=0.50. 
	% It was exucuted through C:\Users\Admin\Desktop\LinearBurger - ENO-L on 24-Jul-2018
\end{table}
\begin{table}[htb!]
	\centering
	\begin{tabular}{|c|c|c|c|c|}
		\hline N  & $L^\infty$ error  &  Rate &    $L^1$ error    & Rate \\ 
		\hline 20 & 0.001301463623011 &  ...  & 0.000619425528920 & ... \\ 
		\hline 40 & 0.000044125068907 &  4.88 & 0.000016956879987 & 5.19 \\ 
		\hline 80 & 0.000001010357188 &  5.45 & 0.000000396311294 & 5.42 \\ 
		\hline 160 & 0.000000028251509 &  5.16 & 0.000000011220498 & 5.14 \\ 
		\hline 320 & 0.000000000875224 &  5.01 & 0.000000000344007 & 5.03 \\ 
		\hline 
	\end{tabular} 
	\caption{Convergence rate of WENO5-$\mathcal{L}$.}
	\label{tab:burger_Acc_WENOL}
	% The solution of test 3 at T=0.16 for CFL=0.50. 
	% It was exucuted through C:\Users\Admin\Desktop\LinearBurger - ENO-L on 24-Jul-2018
\end{table}

\subsection{System of conservation laws: 1D Euler Equations}
The WENO5-$\mathcal{L}$ and WENO5-Z$\mathcal{L}$ schemes are applied to 1D Euler equations \begin{equation}\label{Eulereqn}
\begin{array}{ccc}
\left(\begin{array}{c} 
\rho\\
\rho u \\
E 
\end{array}\right)_t
&+ &
\left(\begin{array}{c} 
\rho u\\
\rho u^{2}+p \\
u(E+p) 
\end{array}\right)_x=0,
\end{array}
\end{equation}
where total energy $E$ and the pressure $p$ are related by the following equation
\begin{equation}
p=(\gamma-1)\left(E-\frac{1}{2}\rho u^{2}\right).
\end{equation}

The result is compared between the solution by WENO5-JS and WENO5-Z for the well known Riemann test problems.
\subsubsection*{Sod shock tube test}
Consider the Sod tube problem which is the Riemann problem \begin{equation}\label{Rie_prb}
(\rho, u, p)(x,0)=\begin{cases}
(\rho_l,\;u_l,\; p_l) & \text{if $x<0$,}\\
(\rho_r,\;u_r,\; p_r) & \text{if $x\geq0$,}
\end{cases}
\end{equation} with $(\rho_l,\;u_l,\; p_l)$ = $(1,\; 0,\; 1)$ and $(\rho_r,\;u_r,\; p_r)$=$(0.125,\; 0,\; 0.1)$. Result are plotted in Figure \ref{SodWENO}. In this case all the results are similar, however, slightly better results can be seen by ENO3-$\mathcal{L}$ in compare to ENO3 in the solution region between the contact discontinuity and shock.
\begin{figure}
	\centering
	\includegraphics[scale=0.55]{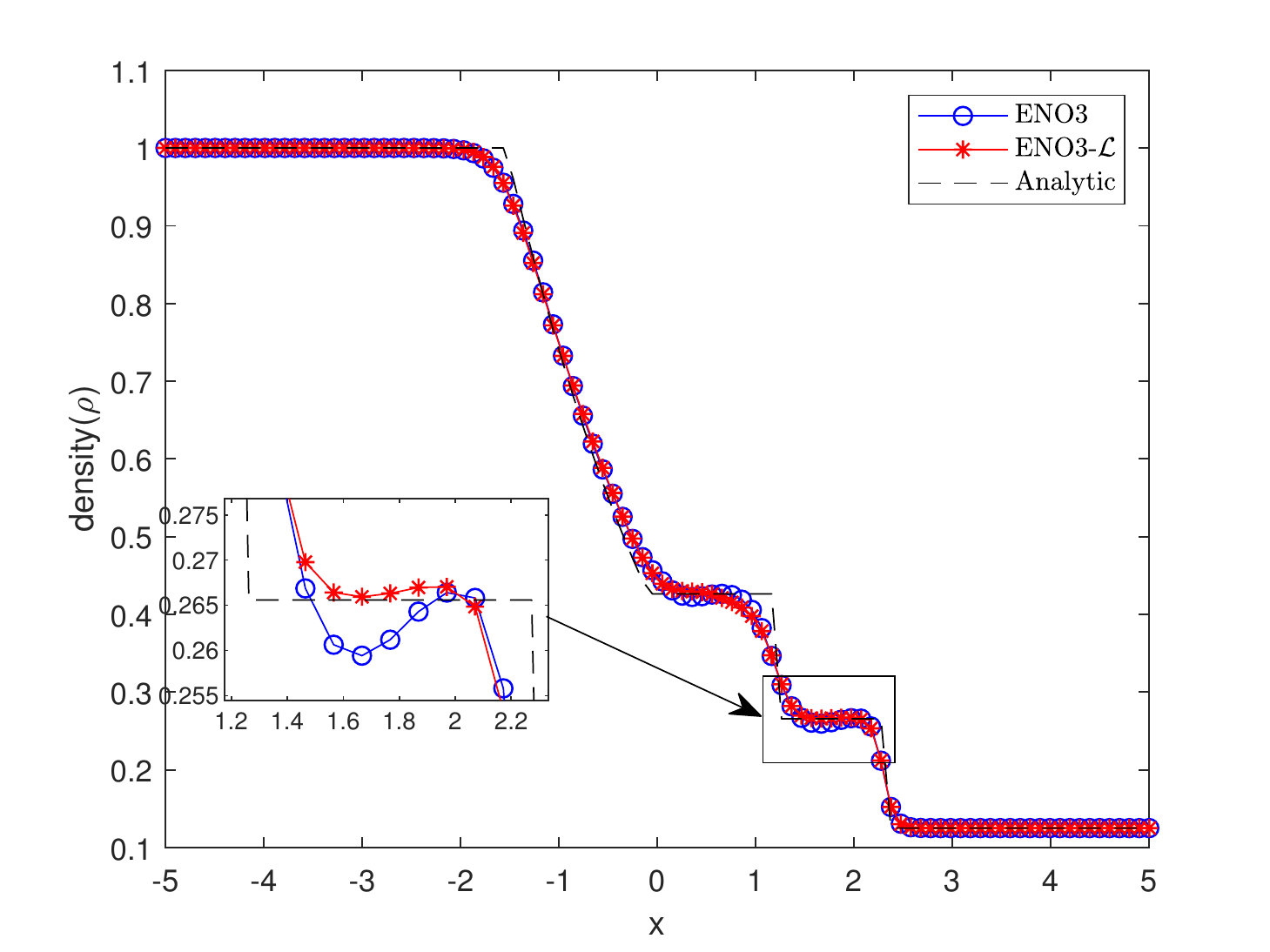}
	\includegraphics[scale=0.55]{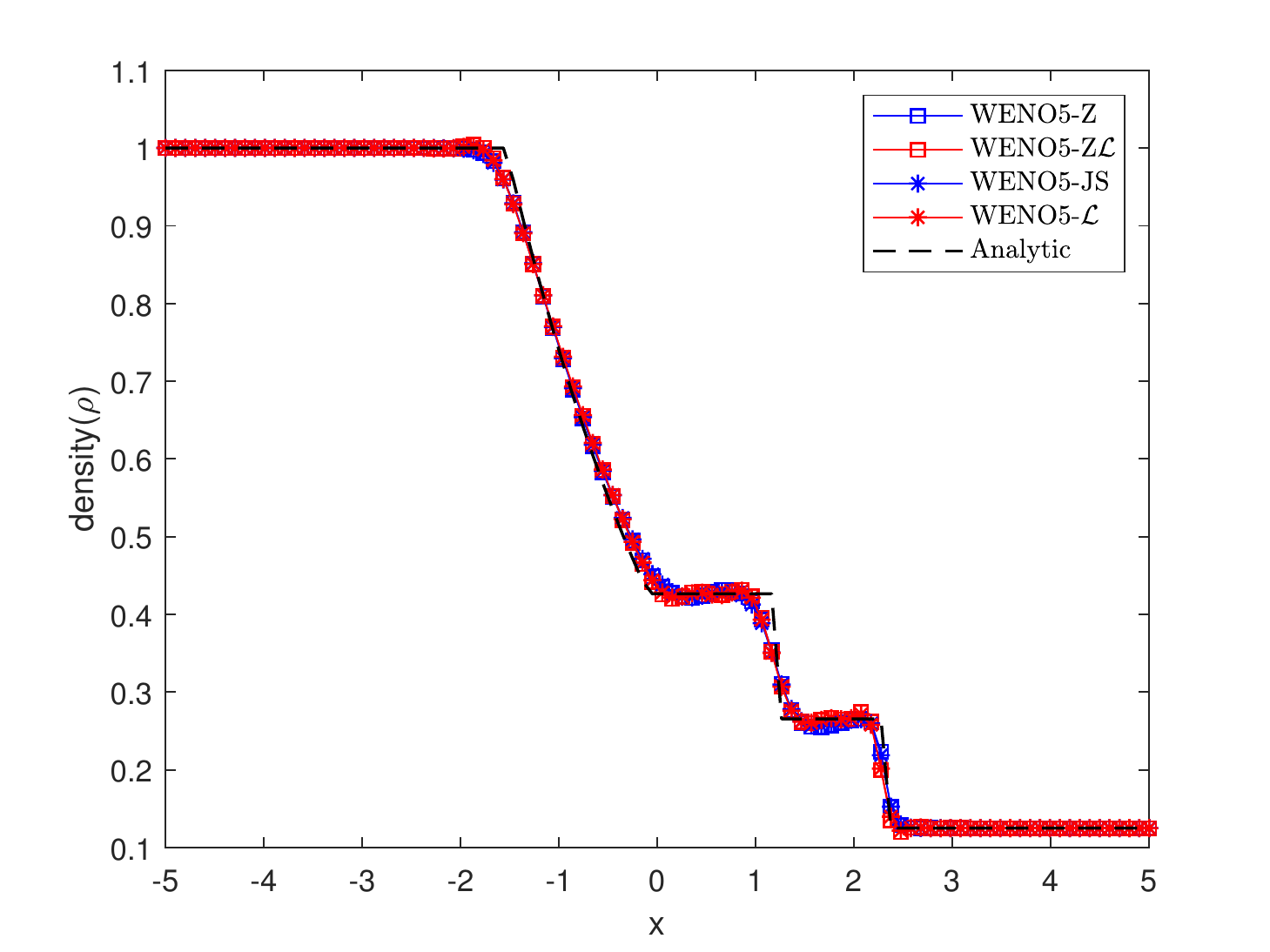}
	\caption{Solution of Sod shock tube test.}
	\label{SodWENO}
\end{figure}
\subsubsection*{Lax tube test}
We consider the Lax tube problem which is to solve Euler 1D equations \eqref{Eulereqn} with the Riemann data \eqref{Rie_prb} where  $(\rho_l,\;u_l,\; p_l)$=$(0.445,\; 0.698,\; 3.528)$ and $(\rho_r,\;u_r,\; p_r)=(0.5,\; 0, \;0.571)$. This test problem contains a comparatively strong shock. Since length of the ENO polynomial has the major role in ignoring the discontinuous region it is expected a better solution near strong shock. WENO5 -$\mathcal{L}$ shows improvement over WENO5-JS and WENO5-Z near strong shocks in Figure \ref{LaxWENO}. ENO3-$\mathcal{L}$ also shows better capturing of the shock.
\begin{figure}
	\centering
	\includegraphics[scale=0.55]{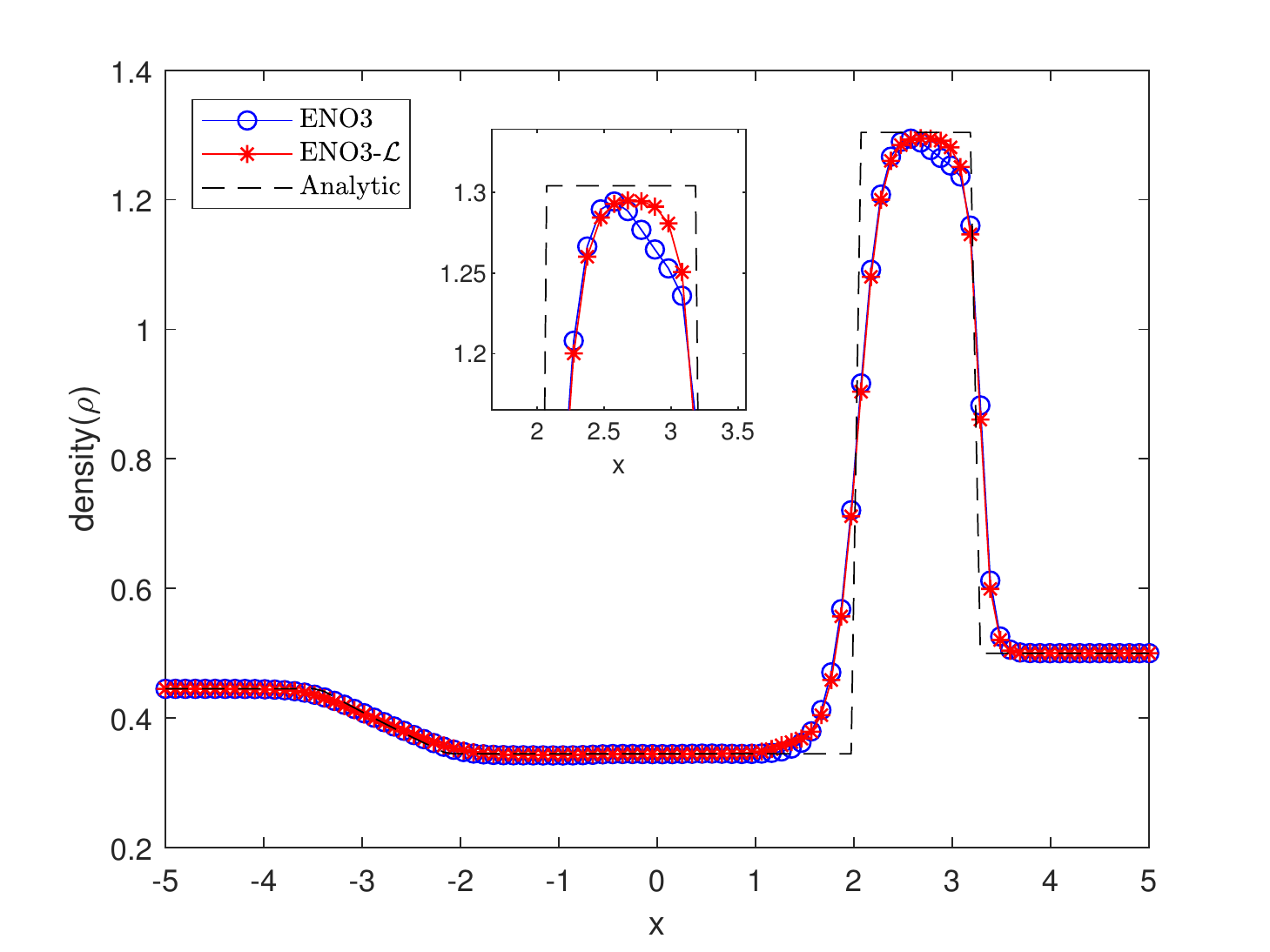}
	\includegraphics[scale=0.55]{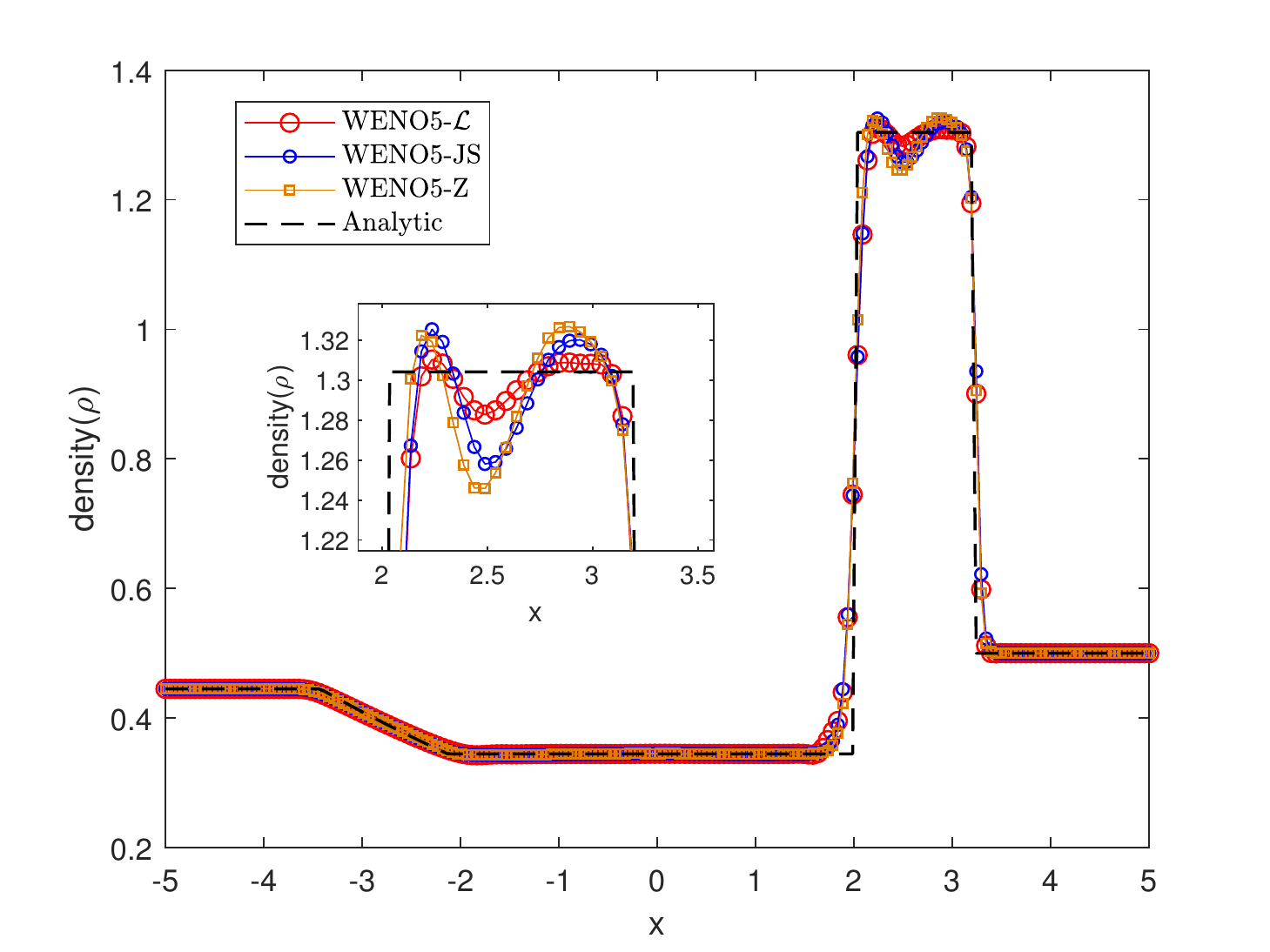}
	\caption{Solution of Lax tube test.}
	\label{LaxWENO}
\end{figure}
\subsection{2D System of conservation laws}
Finally, the scheme is applied to 2D Euler equations \begin{equation}\label{euler2d}
\displaystyle\left(\begin{array}{c}
\rho      \\
\rho u    \\
\rho v   \\
e 
\end{array}\right)_t
+
\displaystyle\left(\begin{array}{c}
\rho u\\
\rho u^2+p\\
\rho uv\\
u(e+p)
\end{array}\right)_x
+\displaystyle\left(\begin{array}{c}
\rho v\\
\rho uv\\
\rho v^2+p\\
v(e+p)
\end{array}\right)_y=0 .
\end{equation}
where $\rho$, $u$, $v$ are density and velocity components along $x$ and $y$
axis respectively. The pressure ($p$) and the energy ($e$) are related by,
\begin{equation}
e=\frac{p}{\gamma-1}+\frac{\rho(u^2+v^2)}{2}.
\end{equation}
The following test problem is considered in order to test the shock capturing property of the scheme in much more complicated problem.

\subsubsection*{2D Riemann problem}
The two dimensional Euler equations is considered with Riemann data, \begin{equation}\label{Rieprb_re}
(p,\rho,u,v)(x,y,0)=\begin{cases}
\text{($p_1$, $\rho_1$, $u_1$, $v_1$),} & \text{if $x>0.5$ and $y>0.5$,}\\
\text{($p_2$, $\rho_2$, $u_2$, $v_2$),} & \text{if $x<0.5$ and $y>0.5$,} \\
\text{($p_3$, $\rho_3$, $u_3$, $v_3$),} & \text{if $x<0.5$ and $y<0.5$,} \\
\text{($p_4$, $\rho_4$, $u_4$, $v_4$),} & \text{if $x>0.5$ and $y<0.5$.}
\end{cases}
\end{equation} 
where the initial data are chosen from \cite{schulz1993numerical} and given in Table \ref{DSinitial}. 

\begin{table}
	\begin{center}
		\begin{tabular}{cccc}
			\hline 
			$ p_1=1.5000 $ & $ p_2=0.3000$ & $p_3=0.0290$ & $ p_4=0.3000$ \\ 
			$ \rho_1=1.5000 $ & $ \rho_2=0.5323$ & $\rho_3=0.1380$ & $ \rho_4=0.5323$ \\ 
			$ u_1=0.0000 $ & $ u_2=1.2060$ & $u_3=1.2060$ & $ u_4=0.0000$ \\ 
			$ v_1=0.0000 $ & $ v_2=0.0000$ & $v_3=1.2060$ & $ v_4=1.2060$ \\ \hline 
		\end{tabular}
		\caption{Initial data for Riemann problem \eqref{Rieprb_re}.}
		\label{DSinitial}
	\end{center}
\end{table}
 Results by WENO5-Z$\mathcal{L}$ is compared with WENO5-JS in Figure \ref{DSinitial}. The modified scheme with new smoothness parameter\eqref{newbeta} captures more characteristics of the flow in Figure \ref{DSinitial} than the WENO5-JS scheme.
\begin{figure}
	\centering
	\includegraphics[scale=0.5]{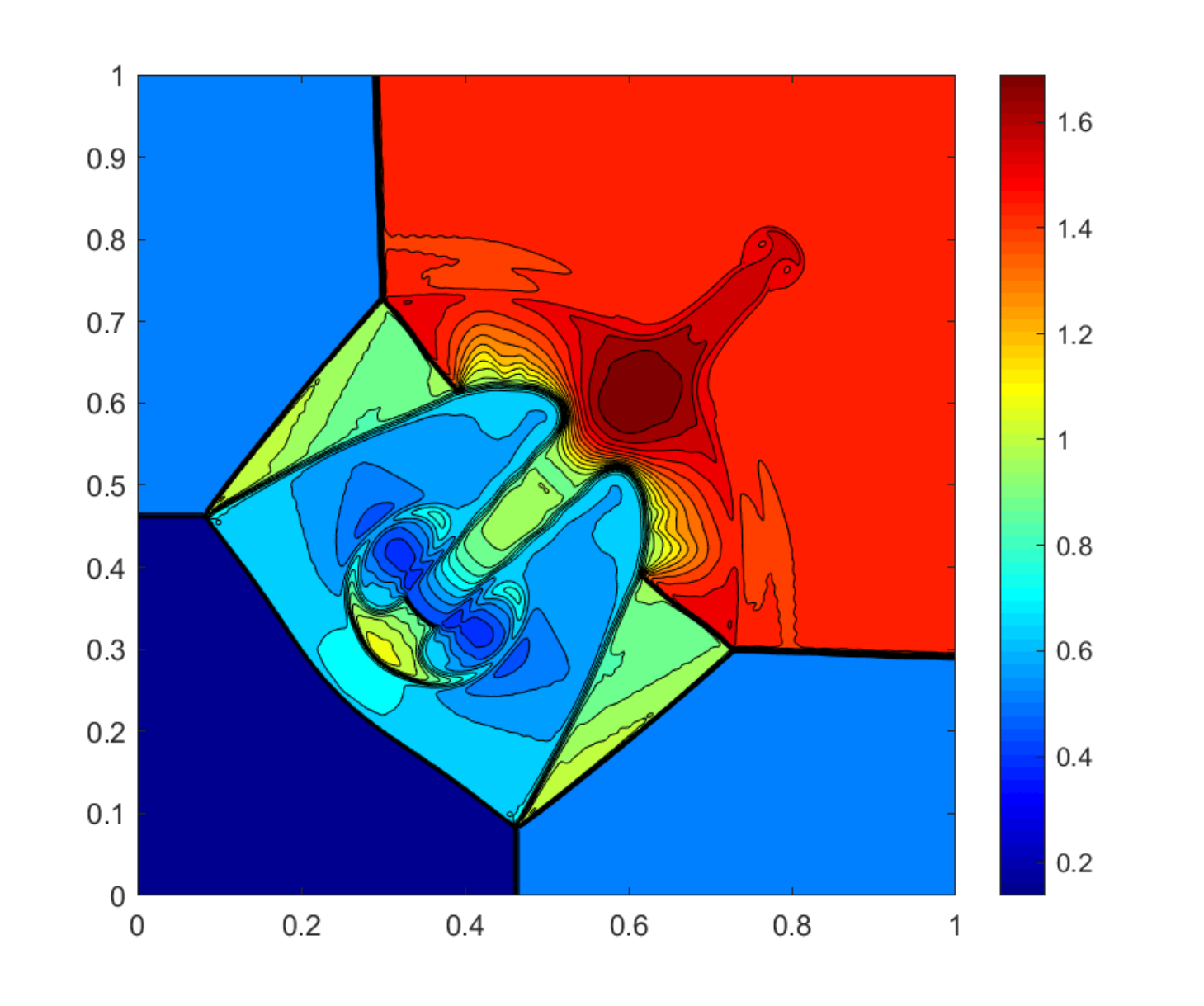}
	\includegraphics[scale=0.5]{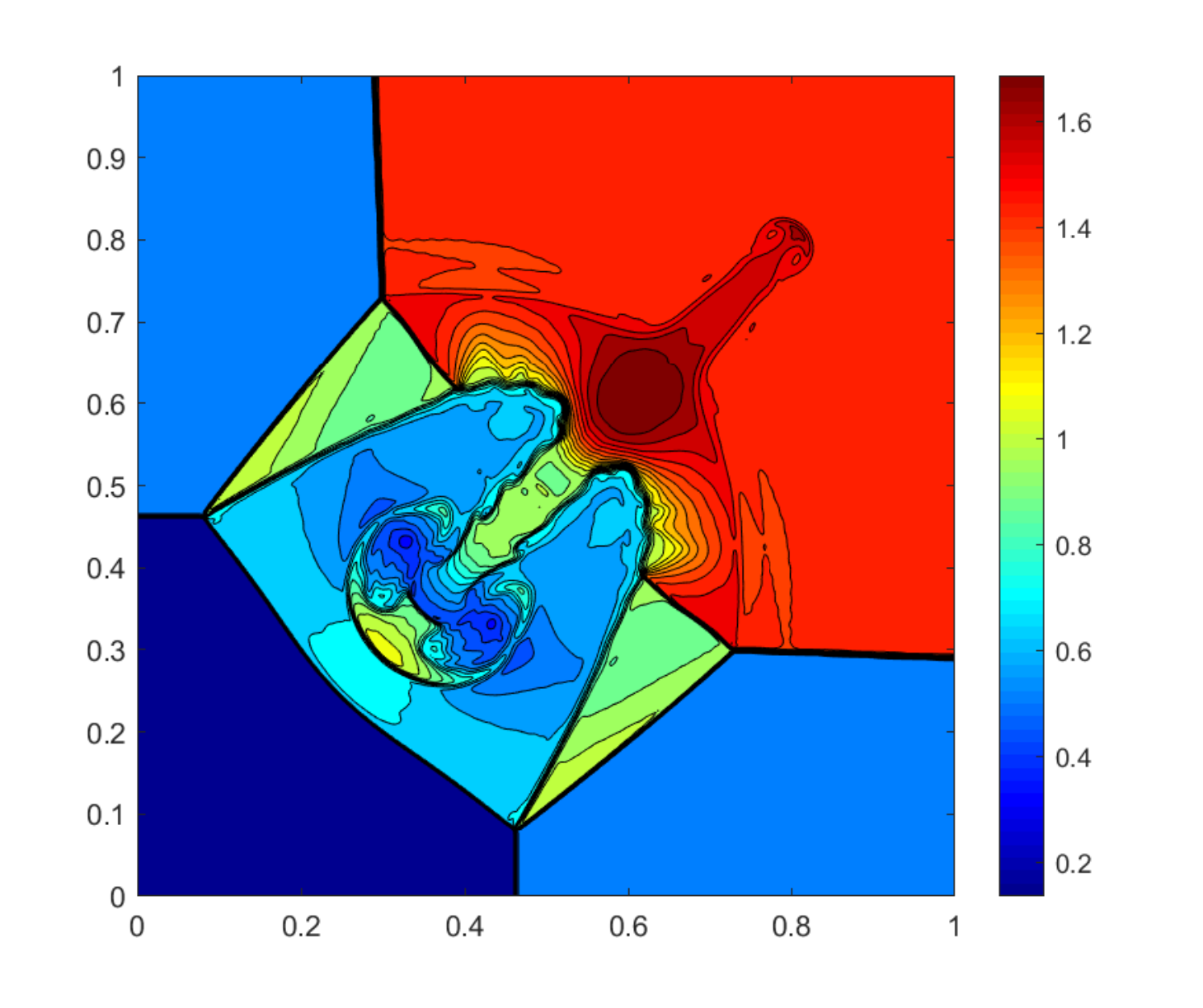}
	\caption[2D Riemann problem by WENO5-JS and WENOZ-$\mathcal{L}$.]{Surface plot of the density in solution profile of \eqref{euler2d} with initial data in Table \ref{DSinitial}. (left) Using scheme WENO5-JS , (right) Using scheme WENO5-Z$\mathcal{L}$}
\end{figure}
\subsubsection*{Implosion Problem}
This is one of the interesting problem from several other problems described in \cite{Hui1999unified}. However, the actual domain and boundary condition is used from \cite{Liska2003comparison}. Gas is kept inside a square domain  $\left[-0.3,0.3\right]\times \left[-0.3,0.3\right]$ in $x-y$ plane. Initial Density and pressure distribution of the gas are following,
\begin{equation*}\label{imp}
\begin{cases}
 \rho (x,y)=0.125, p(x,y)=0.14, & \text{if $|x|+|y|<0.15$ }\\
\rho (x,y)=1, p(x,y)=1. & \text{otherwise }
\end{cases}
\end{equation*}
Initially the velocities are kept zero in the computational domain $\left[0,0.3\right]\times \left[0,0.3\right]$ with reflecting boundary. Notable improvement can be observed in Figure \eqref{Implosion} by WENO5-Z$\mathcal{L}$ over WENO5-JS.
\begin{figure}
	\centering
	\includegraphics[scale=0.5]{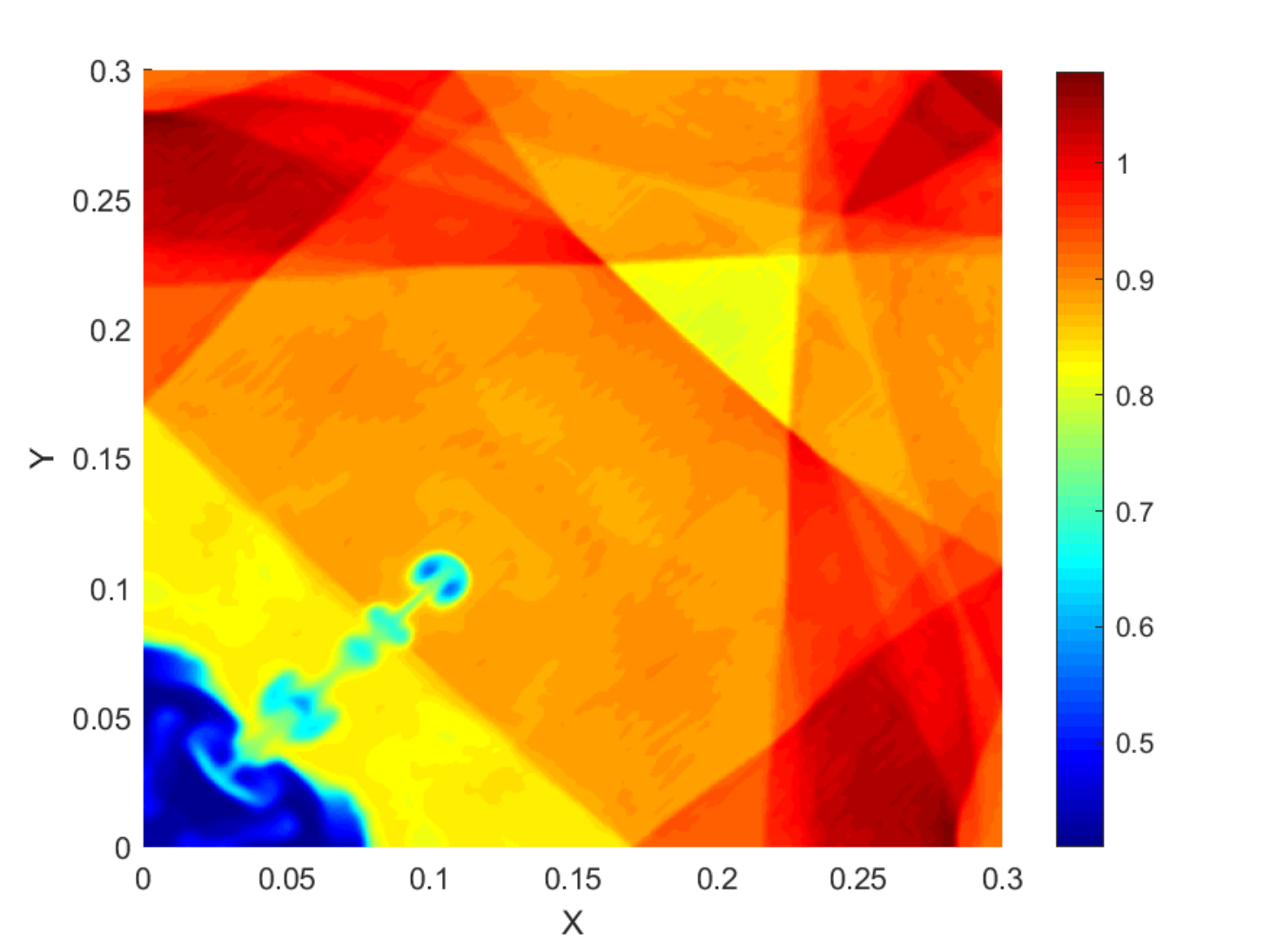}
	\includegraphics[scale=0.5]{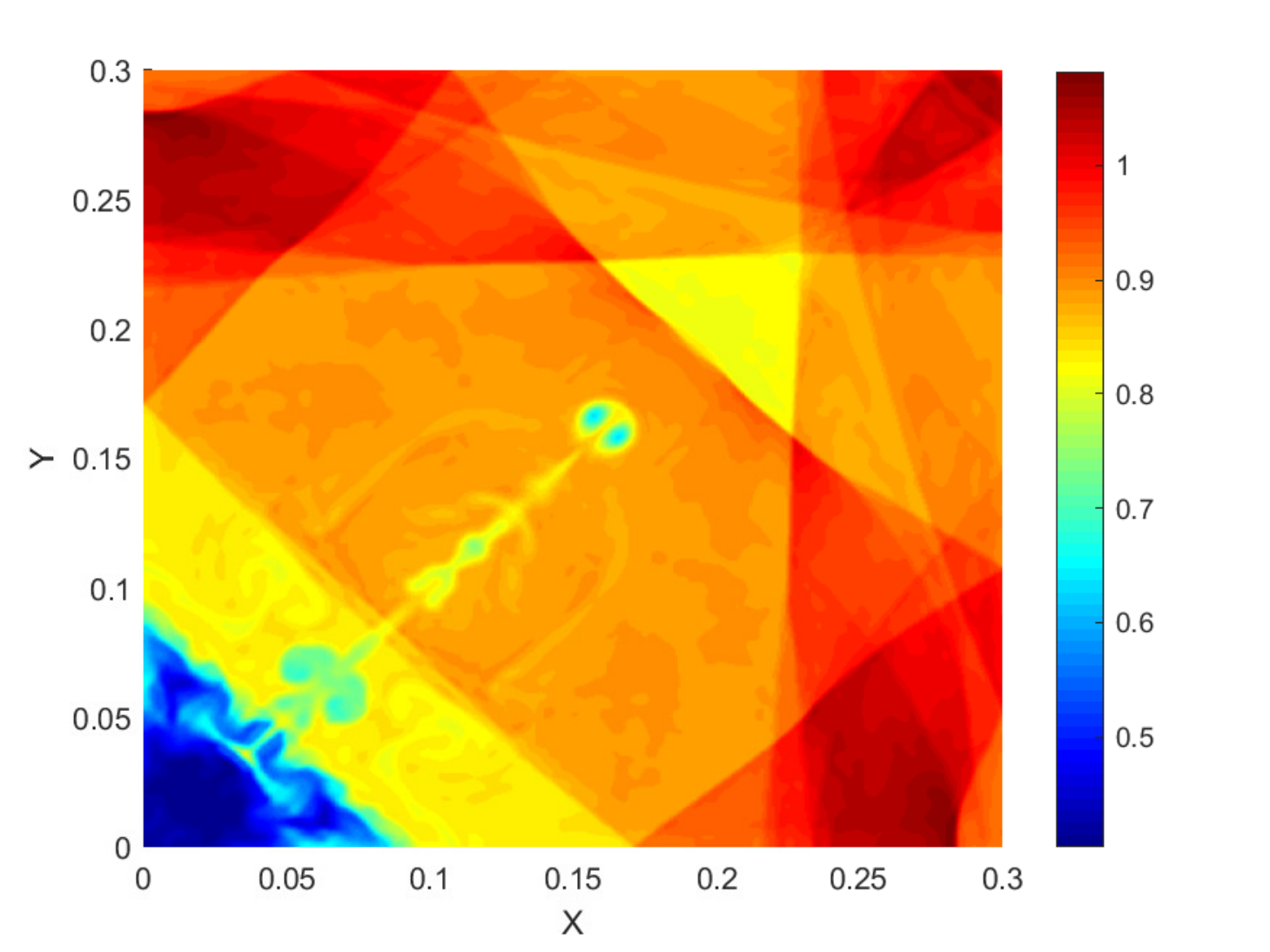}
	\caption[2D Riemann problem by WENO5-JS and WENOZ-$\mathcal{L}$.]{Surface plot of the density in solution profile of \eqref{euler2d} for implosion Problem. (left) Using scheme WENO5-JS , (right) Using scheme WENO5-Z$\mathcal{L}$}
	\label{Implosion}
\end{figure}
\subsubsection*{Explosion Problem}
Explosion test problem is presented in \cite{Toro2009mult}. In this case a square domain $\left[-3,3\right]\times \left[-3,3\right]$ in $x-y$ plane is considered and gas with higher density and pressure is kept inside a circular disc of radius $0.4$ and center $(0,0)$. Initial velocities are kept zero whereas the density and pressure profiles are given by,
\begin{equation*}
\begin{cases}
\rho (x,y)=1, p(x,y)=1, & \text{if $x^2+y^2<(0.4)^2$ }\\
\rho (x,y)=0.125, p(x,y)=0.1. & \text{otherwise }
\end{cases}
\end{equation*}
Computation is done in $\left[0,3\right]\times \left[0,3\right]$ with reflecting boundary. Sharper resolution near discontinuity can be seen in Figure \eqref{Explosion}.
\begin{figure}
	\centering
	\includegraphics[scale=0.5]{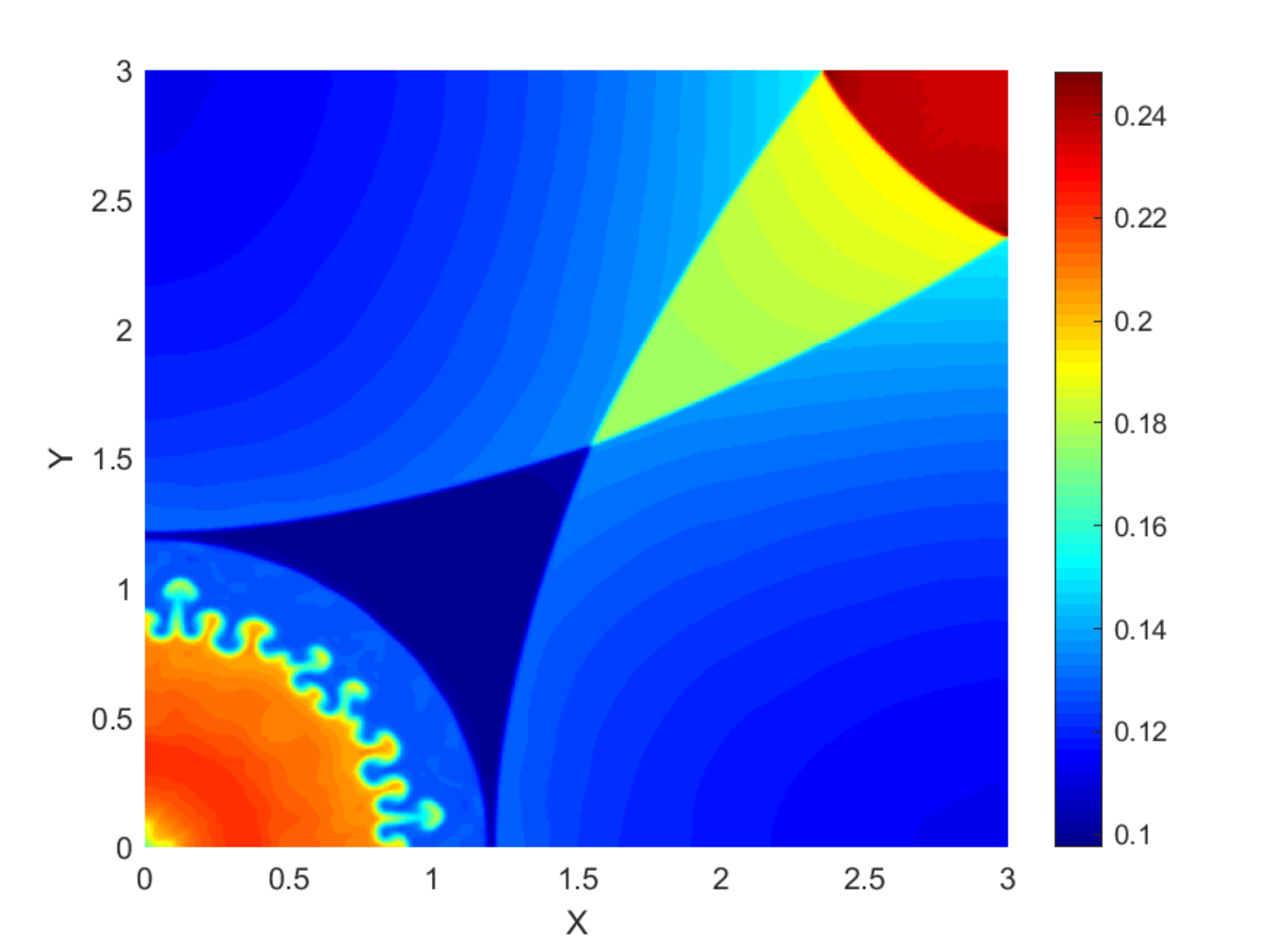}
	\includegraphics[scale=0.5]{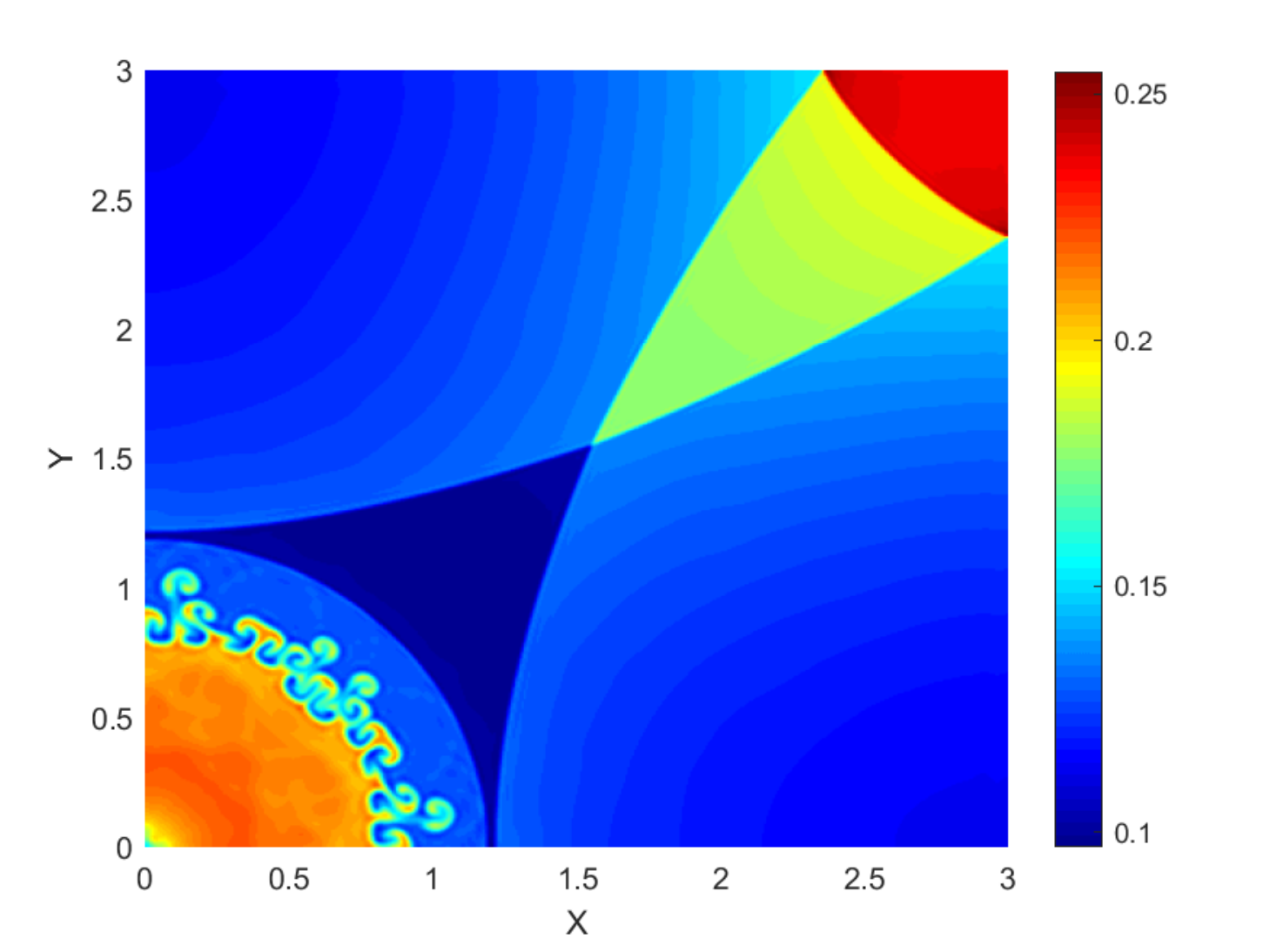}
	\caption[2D Riemann problem by WENO5-JS and WENO5-Z$\mathcal{L}$.]{Surface plot of the density in solution profile of \eqref{euler2d} for explosion Problem. (left) Using scheme WENO5-JS , (right) Using scheme WENO5-Z$\mathcal{L}$}
	\label{Explosion}
\end{figure}
\section*{Final Discussion}
Length of the polynomial curve is used to find the smoothest ENO polynomial in ENO reconstruction. The smoothness indicator of the WENO scheme also defined using the length of the ENO polynomials. The resulting ENO and WENO schemes preserves accuracy at critical points and perform well in capturing shocks with much reduced oscillations. The new WENO schemes namely, WENO5-$\mathcal{L}$ and WENO5-Z$\mathcal{L}$ shown improvements near strong shocks in 1D and 2D test problems.

\section*{Acknowledgement}
Authors acknowledge Council of Scientific and Industrial Research(CSIR), New Delhi, Govt. of India for providing fellowship to Mr. Biswarup Biswas (File no. 09/1045(0009)2K17) and SERB India for its financial support towards procuring computing facility (File No.EMR/2016/000394 )  

\begin{appendices}
	\section{ENO polynomials and computation of their lengths}\label{A1}
	For third order ENO reconstruction, $p^j_i(x)$ (j=0,1,2) in \eqref{ENO_pol} can be explicitly written as \cite{Shu1997}
	\begin{subequations}
		\begin{eqnarray}
		p^0_i(x)=& & - \frac{{(\Delta x)}^2+12 {(\Delta x)} (x_i-x)-12 (x_i-x)^2}{24 {(\Delta x)}^2}\bar{v}_{i-2} \nonumber\\
		& & + \frac{{(\Delta x)}^2+24 {(\Delta x)} (x_i-x)-12 (x_i-x)^2}{12 {(\Delta x)}^2}\bar{v}_{i-1} \nonumber\\
		& & + \frac{23 {(\Delta x)}^2-36 {(\Delta x)} (x_i-x)+12 (x_i-x)^2}{{24 {(\Delta x)}^2}} \bar{v}_i
%		\left(\frac{-{(\Delta x)}^2-12 {(\Delta x)} x_i+12 x_i^2}{24 {(\Delta x)}^2} \bar{v}_{i-2}+\frac{{(\Delta x)}^2+24 {(\Delta x)} x_i-12 x_i^2}{12 {(\Delta x)}^2}\bar{v}_{i-1}+\frac{23 {(\Delta x)}^2-36 {(\Delta x)} x_i+12 x_i^2}{24 {(\Delta x)}^2}\right)
		\end{eqnarray}
		\begin{eqnarray}
		p^1_i(x)=& & -\frac{{(\Delta x)}^2-12 {(\Delta x)} (x_i-x)-12 (x_i-x)^2}{24 {(\Delta x)}^2}\bar{v}_{i-1} \nonumber\\	
			& & + \frac{13 {(\Delta x)}^2-12 (x_i-x)^2}{12 {(\Delta x)}^2}\bar{v}_i \nonumber\\			
			& &- \frac{{(\Delta x)}^2+12 {(\Delta x)} (x_i-x)-12 (x_i-x)^2}{24 {(\Delta x)}^2}\bar{v}_{i+1}
		\end{eqnarray}
		\begin{eqnarray}
		p^2_i(x)=& & \frac{23 {(\Delta x)}^2+36 {(\Delta x)} (x_i-x)+12 (x_i-x)^2}{24 {(\Delta x)}^2}\bar{v}_i\nonumber\\
			& & +\frac{{(\Delta x)}^2-24 {(\Delta x)} (x_i-x)-12 (x_i-x)^2}{12 {(\Delta x)}^2}\bar{v}_{i+1}\nonumber\\
			& & -\frac{{(\Delta x)}^2-12 {(\Delta x)} (x_i-x)-12 (x_i-x)^2}{24 {(\Delta x)}^2}\bar{v}_{i+2}
		\end{eqnarray}
	\end{subequations}
As all the above polynomials are at most second degree they can also be written in the form \begin{equation*}
p^k_i(x)=a_k+b_k x+c_k x^2 \mbox{ for some $a_k, b_k, c_k$; } (k=0,1,2)
\end{equation*}
Let us denote \begin{equation*}
\mathcal{I}[b,c,z]:=\begin{cases}
\displaystyle \frac{(b+2 c z) \sqrt{(b+2 c z)^2+1}+\sinh ^{-1}(b+2 c z)}{4 c} & \text{if $c\neq0$}\\
\sqrt{1+b^2}z & \text{if $c=0$}
\end{cases}
\end{equation*}
Then the explicit lengths are given by
\begin{equation*}
\mathcal{L}_{[x_{i-\frac{1}{2}},\,x_{i+\frac{1}{2}}]}(p_i^k)=\mathcal{I}[b_k,c_k,x_{i+\frac{1}{2}}]-\mathcal{I}[b_k,c_k,x_{i-\frac{1}{2}}],\,\, (k=0,1,2).
\end{equation*}
\end{appendices}
\pagebreak
\bibliographystyle{alpha}
\bibliography{mybib}
\end{document}